\documentclass{llncs}
\usepackage{makeidx}  

\usepackage{amsmath,amsopn,amssymb}
\usepackage{etex}
\usepackage{bm}
\DeclareMathOperator{\Expectation}{\mathbb E} 
\newcommand{\Ccs}[1]{C_c\left(#1\right)}
\newcommand{\Cexp}[1]{C_c^{(\cosh-1)}\left(#1\right)}
\newcommand{\Cinfcs}[1]{C_c^\infty\left(#1\right)}
\newcommand{\Lexp}[1]{L^{(\cosh-1)}\left(#1\right)}
\newcommand{\LlogL}[1]{L^{(\cosh-1)_*}\left(#1\right)}
\newcommand{\absoluteval}[1]{\left|#1\right|}
\newcommand{\euler}{\mathrm{e}}
\newcommand{\expectat}[2]{{\Expectation}_{#1}\left[#2\right]}
\newcommand{\expectof}[1]{\Expectation\left(#1\right)}
\newcommand{\expof}[1]{\exp\left(#1\right)}
\newcommand{\naturals}{\mathbb N}
\newcommand{\normat}[2]{\left\Vert#2\right\Vert_{#1}}
\newcommand{\one}{\bm 1}
\newcommand{\probabilities}{\mathcal P}
\newcommand{\reals}{\mathbb R}
\newcommand{\scalarat}[3]{\left\langle#2,#3\right\rangle_{#1}}
\newcommand{\sdomain}[1]{\mathcal S_{#1}}
\newcommand{\setof}[2]{\left\{#1 \middle| #2 \right\}}
\newcommand{\set}[1]{\left\{#1\right\}}

\begin{document}

\mainmatter              
\title{Translations in the exponential Orlicz space with Gaussian weight}
\titlerunning{Translations on $\Lexp M$}  
%
\author{Giovanni Pistone}
\authorrunning{G. Pistone}
\institute{de Castro Statistics, Collegio Carlo Alberto, Moncalieri, Italy\\
\email{giovanni.pistone@carloalberto.org}\\
\texttt{www.giannidiorestino.it}
}

\maketitle              

\begin{abstract}
We study the continuity of space translations on non-parametric exponential families based on the exponential Orlicz space with Gaussian reference density.\end{abstract}

\section{Introduction}
\label{sec:introduction}
On the Gaussian probability space $(\reals^n,\mathcal B,M \cdot \ell)$, $M$ being the standard Gaussian density and $\ell$ the Lebesgue measure, we consider densities of the form $\euler_M(U) = \expof{U - K_M(U)} \cdot M$, where $U$ belongs to the exponential Orlicz space $\Lexp M$, $\expectat M U = 0$, and $K_M(U)$ is constant  \cite{pistone|sempi:95,pistone:2013GSI}. An application to the homogeneous Boltzmann equation has been discussed in \cite{lods|pistone:2015}.

The main limitation of the standard version of Information Geometry is its inability to deal with the structure of the sample space as it provides a geometry of the ``parameter space'' only. As a first step to overcome that limitation, we want to study the effect of a space translation  $\tau_h$, $h \in \reals^n$, on the exponential probability density $\euler_M(U)$. Such a model has independent interest and, moreover, we expect such a study to convey informations about the case where the density $\euler_M(U)$ admits directional derivatives.

The present note is devoted to the detailed discussion of the some results concerning the translation model that have been announced at the IGAIA IV Conference, Liblice CZ on June 2016. All results are given in Sec.~\ref{sec:gauss-orlicz-spaces}, in particular the continuity result in Prop. \ref{prop:taubyh}. The final Sec.~\ref{sec:conclusions} gives some pointers to further research work to be published elsewhere.

\section{Gauss-Orlicz spaces and translations}
\label{sec:gauss-orlicz-spaces}

The \emph{exponential space} $\Lexp M$ and the \emph{mixture space} $\LlogL M$ are the Orlicz spaces associated the Young functions $(\cosh-1)$ and its convex conjugate $(\cosh - 1)_*$, respectively \cite{musielak:1983}. They are both Banach spaces and the second one has the $\Delta_2$-property, because of the inequality
\begin{equation*}
  (\cosh-1)_*(ay) \le \max(1,a^2) (\cosh-1)_*(y),\quad a,y \in \reals \ .
\end{equation*}

The closed unit balls are $$\setof{f}{\int \phi(f(x)) \ M(x)dx \le 1}$$ with $\phi=\cosh-1$ and $\phi=(\cosh-1)_*$, respectively. Convergence to 0 in norm of a sequence $g_n$, $n\in\naturals$ holds if, and only if, for all $\rho > 0$ one has
\begin{equation*}
\limsup_{n \to \infty} \int \phi(\rho g_n(x)) \ M(x)dx \le 1 \ .
\end{equation*}

If $1 < a < \infty$, the following inclusions hold
      \begin{equation*}
        L^\infty(M) \hookrightarrow \Lexp M \hookrightarrow L^a(M) \hookrightarrow \LlogL M \hookrightarrow L^1(M) \ ,
      \end{equation*}
and the restrictions to the ball $\Omega_R = \setof{x \in \reals^n}{\absoluteval x < R}$,
      \begin{equation*}
        \Lexp M \rightarrow L^a(\Omega_R), \quad  \LlogL M \rightarrow L^1(\Omega_R) \ ,
      \end{equation*}
      are continuous.
      
      The exponential space $\Lexp M$ contains all functions $f \in C^2(\reals^n;\reals)$ whose Hessian is uniformly bounded in operator's norm. In particular, it contains all polynomials with degree up to 2, hence all functions which are bounded by such a polynomial. The mixture space $\LlogL M$ contains all random variables $f \colon \reals^d \to \reals$ which are bounded by a polynomial, in particular, all polynomials.

Let us review those properties of the exponential function on the space $\Lexp M$ that justify our definition of \emph{non-parametric exponential} model as the set of densities $\euler_M(U) = \expof{U - K_M(U)} \cdot M$, where $U$ has zero $M$-expectation and belongs to the interior $\sdomain M$ of the proper domain of the partition functional $Z_M(U) = \expectat M {\euler^U}$.
\begin{proposition}
  \begin{enumerate}
  \item \label{item:1} The functionals $Z_M$ and $K_M = \log Z_M$ are both convex.
  \item \label{item:2} The proper domain of both $Z_M$ and $K_M$ contains the open unit ball of $\Lexp M$, hence its interior $\sdomain M$ is nonempty.
  \item \label{item:3} The functions $Z_M$ and $K_M$ are both Fr\'echet differentiable on $\sdomain M$.
  \end{enumerate}
\end{proposition}
  \begin{proof} Statements \ref{item:1}--\ref{item:3} above are all well known. Nevertheless, we give the proof of the differentiability. We have
\begin{equation*}
0 \le \expof {U+H} - \expof U - \expof U H = \int_0^1 (1-s) \expof{U+sH} H^2 \ ds \ .
\end{equation*}
For all $U,U+H \in \sdomain M$, choose $\alpha > 1$ such that $\alpha U \in \sdomain M$. We have
\begin{equation*}
  0 \le Z_M(U+H) - Z_M(U) - \expectat M {\expof U H} = \int_0^1 (1-s) \expectat M {\expof{U+sH} H^2} \ ds \ ,
\end{equation*}
where the derivative term $H \mapsto \expectat M {\expof U H}$ is continuous at $U$ because
\begin{multline*}
  \absoluteval{\expectat M {\expof U H}} \le \expectat M {\expof {\alpha U}}^{1/\alpha} \expectat M {\absoluteval H^{\alpha/(\alpha-1)}}^{(\alpha-1)/\alpha} \le \\ \text{const} \times \  \expectat M {\expof {\alpha U}}^{1/\alpha} \normat {\Lexp M} H \ . 
\end{multline*}
The remainder term is bounded by
\begin{multline*}
  \absoluteval{Z_M(U+H) - Z_M(U) - \expectat M {\expof U H}} = \\ \int_0^1 (1-s) \expectat M {\expof{U+sH} H^2} \ ds \le \\
  \expectat M {\euler^{\alpha U}}^{1/\alpha} \int_0^1 (1-s) \expectat M {\expof{s\frac{\alpha}{\alpha -1}H} H^{2\frac{\alpha}{\alpha-1}}}^{(\alpha-1)/\alpha}
  \ ds \le \\ \text{const} \times \expectat M {H^{4\frac{\alpha}{\alpha-1}}}^{(\alpha-1)/2\alpha}
  \int_0^1 (1-s) \expectat M {\expof{s\frac{2\alpha}{\alpha -1}H}}^{(\alpha-1)/2\alpha} 
  \ ds
  \ .
\end{multline*}
We have
\begin{equation*}
\expectat M {\expof{s\frac{2\alpha}{\alpha -1}H}} \le 2\left(\expectat M {(\cosh-1)\left(s\frac{2\alpha}{\alpha -1}H\right)+1}\right) \le 4 
\end{equation*}
if $\normat {\Lexp M}{H} \le (\alpha-1)/2\alpha$. Under this condition, we have
\begin{multline*}
 \absoluteval{Z_M(U+H) - Z_M(U) - \expectat M {\expof U H}} \le \\ \text{const}\times \normat{L^{4\alpha/(\alpha-1)}(M)} H ^2 \le \text{const}\times \normat {\Lexp M} H ^2 \ ' 
\end{multline*}
where the constant depends on $U$. \qed
\end{proof}

The space $\Lexp M$ is neither separable nor reflexive. However, we have the following density property for the bounded point-wise convergence. The proof uses a form of the Monotone-Class argument \cite[22.3]{dellacherie|meyer:75}. Let $\Ccs{\reals^n}$ and $\Cinfcs{\reals^n}$ respectively denote the space of continuous real functions with compact support and its sub-space of infinitely-differentiable functions.

\begin{proposition}\label{prop:density} 
For each $f \in \Lexp M$ there exists a nonnegative function $h \in \Lexp M$ and a sequence $f_n \in \Cinfcs{\reals^n}$ with $\absoluteval{f_n} \le h$, $n=1,2,\dots$, such that $\lim_{n\to\infty} f_n = f$ a.e. As a consequence, $\Cinfcs {\reals^n}$ is weakly dense in $\Lexp M$. 
\end{proposition}
\begin{proof}
Before starting the proof, let us note that $\Lexp M$ is stable under bounded a.e. convergence. Assume $f_n,h \in \Lexp M$ with $\absoluteval{f_n} \le h$, $n=1,2,\dots$ and $\lim_{n\to\infty} f_n = f$ a.e. By definition of $h \in \Lexp M$, for $\alpha = \normat {\Lexp M} h^{-1}$ we have the bound $\expectat M {(\cosh-1)(\alpha h)} \le 1$. The sequence of functions $(\cosh-1)(\alpha f_n)$, $n = 1,2,\dots$, is a.e. convergent to $(\cosh-1)(\alpha f)$ and it is bounded by the integrable function $(\cosh-1)(\alpha h)$. The inequality $\expectat M {(\cosh-1)(\alpha f)} \le 1$ follows now by dominated convergence and is equivalent to $\normat {\Lexp M} f \le \normat {\Lexp M} h$. By taking a converging sequences $(f_n)$ in  $\Cinfcs {\reals^n}$ we see that the condition in the proposition is sufficient. Conversely, let $\mathcal L$ be the set of all functions $f \in \Lexp M$ such that there exists a sequence $(f_n)_{n\in\naturals}$ in $C_c(\reals^n)$ which is dominated by a function $h \in \Lexp M$ and converges to $f$ point-wise. The set $\mathcal L$ contains the constant functions and $C_c(\reals^n)$ itself. The set $\mathcal L$ is a vector space: if $f^1,f^2 \in \mathcal L$ and both $f^1_n \to f^1$ a.s. with $\absoluteval{f^1_n} \le h^1$ and $f^2_n \to f^2$ point-wise with $\absoluteval{h^2_n} \le h^2$, then $\alpha_1 f^1_n + \alpha_2 f^2_n \to \alpha_1 f^1 + \alpha_2 f^2$ point-wise with $\absoluteval{\alpha_1 f^1_n + \alpha_2 f^2_n} \le \absoluteval{\alpha_1} h^1 + \absoluteval{\alpha_2} h^2$. Moreover, $\mathcal L$ is closed under the $\min$ operation: if $f^1,f^2 \in \mathcal L$, with both $f^1_n \to f^1$ with $\absoluteval{g^1_n} \le h^1$ and $f^2_n \to f^2$ with $\absoluteval{g^2_n} \le h^2$, then $f^1_n \wedge f^2_n \to f_1 \wedge f_2$ and $\absoluteval{f^1_n \wedge f^2_n} \le h^1 \wedge h^2 \in \Lexp M$. $\mathcal L$ is closed for the maximum too, because $f^1 \vee f^2 = - \left((-f^1) \wedge (-f^2)\right)$. We come now to the application of  the Monotone-Class argument. As $\one_{f > a} = ((f-a)\vee 0)\wedge 1 \in \mathcal L$, each element of $\mathcal L$ is the point-wise limit of linear combinations of indicator functions in $\mathcal L$. Consider the class $\mathcal C$ of sets whose indicator belongs to $\mathcal L$. $\mathcal C$ is a $\sigma$-algebra because of the closure properties of $\mathcal L$ and contains all open bounded rectangles of $\reals^n$ because they are all of the form $\set{f > 1}$ for some $f \in \Ccs {\reals^n}$. Hence $\mathcal C$ is the Borel $\sigma$-algebra and $\mathcal L$ is the set of Borel functions which are bounded by an element of $\Lexp M$, namely $\mathcal L = \Lexp M$. To conclude, note that each $g \in \Ccs {\reals^n}$ is the uniform limit of a sequence in $\Cinfcs{\reals^n}$. The last statement is proved by bounded convergence. \qed
\end{proof}

Let us discuss some consequences of this result. Let be given $u \in \sdomain M$ and consider the exponential family $p(t) = \expof{tu - K_M(tu)} \cdot M$, $t \in ]-1,1[$. From Prop. \ref{prop:density} we get a sequence $(f_n)_{n\in\naturals}$ in $\Cinfcs{\reals^n}$ and a bound $h \in \Lexp M$ such that $f_n \to u$ point-wise and $\absoluteval{f_n}, \absoluteval{u} \le h$. As $\sdomain M$ is open and contains 0, we have $\alpha h \in \sdomain M$ for some $0 < \alpha < 1$. For each $t \in ]-\alpha,\alpha[$, $\expof{tf_n} \to \expof{tu}$ point-wise and $\expof{tf} \le \expof{\alpha h}$ with $\expectat M {\expectof{\alpha h}} < \infty$. It follows that $K_M(tf_n) \to K(tu)$, so that we have the point-wise convergence of the density $p_n(t) = \expof{tf_n-K_M(tf_n)} \cdot M$ to the density $p(t)$. By Scheff\'e's lemma, the convergence holds in $L^1(\reals^n)$. In particular, for each $\phi \in \Cinfcs{\reals^n}$, we have the convergence
\begin{equation*}
  \int \partial_i \phi(x) p_n(x;t) \ dx \to \int \partial_i \phi(x) p(x;t) \ dx, \quad n\to\infty \ .
\end{equation*}
for all $t$ small enough. By computing the derivatives, we have
\begin{multline*}
 \int \partial_i \phi(x) p_n(x;t) \ dx = - \int \phi(x) \partial_i\left(\euler^{tf_n(x) - K_M(tf_n)} M(x)\right) \ dx = \\ \int \phi(x) \left(x_i - t \partial_i f_n(x)\right) p_n(x;t) \ dx  \ ,
\end{multline*}
that is, $$\left(X_i - t\partial_if_n\right)p_n(t) \to -\partial_ip(t)$$ in the sense of (Schwartz) distributions.
It would be of interest to discuss the possibility of the stronger convergence of $p_n(t)$ in $\LlogL M$, but we do follow this development here. 

The norm convergence of the point-wise bounded approximation will not hold in general. Consider the following example. The function $f(x) = \absoluteval x^2$ belongs in $\Lexp M$, but for the tails $f_R(x) = (\absoluteval x > R) \absoluteval x^2$ we have
\begin{equation*}
  \int (\cosh-1)(\epsilon^{-1} f_R(x)) \ M(x)dx \ge \frac12 \int_{\absoluteval x > R} \euler^{\epsilon^{-1} \absoluteval x^2} \ M(x)dx = +\infty, \quad \text{if $\epsilon \le 2$} \ ,
\end{equation*}
hence there is no convergence to 0. However, the truncation of $f(x) = \absoluteval x$ does converge. This, together with Prop. \ref{prop:density}, suggests the following variation of the classical definition of Orlicz class.

\begin{definition}
The \emph{exponential class}, $\Cexp M$, is the closure of $\Cinfcs{\reals^n}$ in the space $\Lexp M$.
\end{definition}  
\begin{proposition}\label{prop:d}
 Assume $f \in \Lexp M$ and write $f_R(x) = f(x)(\absoluteval x > R)
$. The following conditions are equivalent: 
\begin{enumerate}
\item\label{item:d1}  The real function $\rho \mapsto \int (\cosh-1)(\rho f(x)) \ M(x)dx$ is finite for all $\rho > 0$.
\item\label{item:d2} $f$ is the limit in $\Lexp M$-norm of a sequence of bounded functions.
\item\label{item:d3}  $f \in \Cexp M$.
\end{enumerate}
\end{proposition}

\begin{proof}
  \begin{description}
  \item[$\eqref{item:d1}\Leftrightarrow\eqref{item:d2}$] This is well known, but we give a proof for sake of clarity. We can assume $f \ge 0$ and consider the sequence of bounded functions $f_n = f \wedge n$, $n=1,2,\dots$. We have for all $\rho > 0$ that $\lim_{n\to\infty} (\cosh-1)(\rho (f - f_n)) = 0$ point-wise and $(\cosh-1)(\rho (f - f_n))M \le (\cosh-1)(\rho (f)M$ which is integrable by assumption. Hence 
    \begin{multline*}
    0 \le   \limsup_{n\to\infty} \int (\cosh-1)(\rho (f(x) - f_n(x)))M(x)\ dx \le \\ \int \limsup_{n\to\infty} (\cosh-1)(\rho (f(x) - f_n(x)))M(x)\ dx = 0 \ ,
    \end{multline*}
    which in turn implies $\lim_{n\to\infty} \normat {\Lexp M} {f - f_n} = 0$. Conversely, observe first that we have from the convexity of $(\cosh-1)$ that 
  \begin{equation*}
    2(\cosh-1)(\rho(x+y)) \le (\cosh-1)(2\rho x)+(\cosh-1)(2\rho y) \ .
  \end{equation*}
  It follows that, for all $\rho > 0$ and $n=1,2,\dots$, we have
  \begin{multline*}
   2  \int (\cosh-1)(\rho f(x))M(x)\ dx \le \\ \int (\cosh-1)(2\rho(f(x)-f_n(x)))M(x)\ dx + \int (\cosh-1)(2\rho f_n(x))M(x)\ dx \ ,
  \end{multline*}
  where the $\limsup _{n\to\infty}$ of the first term of the RHS is bounded by 1 because of the assumption of strong convergence, while the second term is bounded by $(\cosh-1)(2\rho n)$. Hence the LHS is finite for all $\rho > 0$.
\item[$\eqref{item:d2}\Rightarrow\eqref{item:d3}$] Assume first $f$ bounded and use Prop. \ref{prop:density} to find a point-wise approximation $f_n \in C_0(\reals^n)$, $n \in \naturals$, of $f$ together with a dominating function $|f_n(x)| \le h(x)$, $h \in \Lexp M$. As $f$ is actually bounded, we can assume $h$ to be equal to the constant bounding $f$. We have $\lim_{n\to\infty}(\cosh-1)(\rho(f-f_n)) = 0$ point-wise, and $(\cosh-1)(\rho(f-f_n)) \le (\cosh-1)(2 \rho h)$. By dominated convergence we have $\lim_{n\to\infty}\int (\cosh-1)(\rho(f(x)-f_n(x)))M(x)\ dx = 0$ for all $\rho>0$, which implies the convergence $\lim_{n\to\infty} \normat {\Lexp M}{f-f_n} = 0$. Because of \eqref{item:d2}, we have the desired result.
\item[$\eqref{item:d3}\Rightarrow\eqref{item:d2}$] Obvious from $\Ccs {\reals^n} \subset L^\infty(M)$.\qed
\end{description}
\end{proof}

We discuss now properties of translation operators in a form adapted to the exponential space $\Lexp M$. Define $\tau_h f(x) = f(x - h)$, $h\in\reals^n$.

\begin{proposition}[Translation by a vector]\
\label{prop:taubyh}  
\begin{enumerate}
\item \label{item:taubyh1} For each $h \in \reals^n$, the mapping $f \mapsto \tau_h f$ is linear from $\Lexp M$ to itself and $\normat {\Lexp M} {\tau_h f} \le 2 \normat {\Lexp M} f$ if $\absoluteval h \le \sqrt{\log 2}$.

\item \label{item:taubyh3} The transpose of $\tau_h$ is defined on $\LlogL M$ by $\scalarat M {\tau_h f} g = \scalarat M f {\tau_h^* g}$, $f \in \Lexp M$, and is given by $\tau_h^* g(x) = \euler^{-h \cdot x + \absoluteval h^2/2} \tau_{-h}g(x)$. For the dual norm, the bound $\normat {\Lexp M^*} {\tau^*_h g} \le 2 \normat {\Lexp M^*} g$ holds if $\absoluteval h \le \sqrt {\log 2}$.

\item \label{item:taubyh2} If $f \in \Cexp M$ then $\tau_h f \in \Cexp M$, $h \in \reals^n$ and the mapping $\reals^n \colon h \mapsto \tau_h f$ is continuous in $\Lexp M$. 
  \end{enumerate}
\end{proposition}

\begin{proof}
\begin{enumerate}
\item Let us first prove that $\tau_h f \in \Lexp M$. It is enough to consider the case $\normat {\Lexp M} f \le 1$. For each $\rho > 0$, with $\Phi = \cosh-1$, we have
      \begin{equation*}
        \int \Phi(\rho \tau_h f(x))\ M(x)dx = \euler^{ - \frac12 \absoluteval h^2} \int \euler^{- z\cdot h} \Phi(\rho f(z))\  M(z)dz \ ,
      \end{equation*}
hence, using the elementary inequality $\Phi(u)^2 \le \Phi(2u)/2$, we obtain
\begin{multline*}
  \int \Phi(\rho\tau_h f(x)) \ M(x)dx \le \\ \euler^{ - \frac12 \absoluteval h^2} \left(\int \euler^{- 2z\cdot h}\ M(z)dz\right)^{\frac12} \left(\int \Phi^2(\rho f(z))\ M(z)dz\right)^{\frac12} \le \\ \frac1{\sqrt 2}\euler^{\frac{\absoluteval h^2}2} \left(\int \Phi(2\rho f(z)) M(z)\ dz\right)^{\frac12} \ .
\end{multline*}
Take $\rho=1/2$ to get $\expectat M {\Phi\left(\tau_h \frac12 f(x)\right)} \le \euler^{\frac{\absoluteval h^2}2}/\sqrt 2$, which in particular implies $f \in \Lexp M$. Moreover, $\normat {\Lexp M}{\tau_h f} \le 2$ if $\euler^{\frac{\absoluteval h^2}2} \le \sqrt 2$.

\item The computation of $\tau^*_h$ is
\begin{multline*}
  \scalarat M {\tau_h f} g = \int f(x-h)g(x) \ M(x)dx = \int f(x) g(x+h)M(x+h) \ dx \\
= \int f(x) \euler^{-h\cdot x - \frac{\absoluteval h^2}2} \tau_{-h}g(x) \ M(x)dx
= \scalarat M f {\tau^*_h g} \ .
\end{multline*}

If $\absoluteval h \le \sqrt{\log 2}$,
\begin{multline*}
\normat {(\Lexp M)^*} {\tau^*_h g} = \sup\setof{\scalarat M {\tau_h f} g}{\normat {\Lexp M} f \le 1} \le \\ \sup\setof{\normat {\Lexp M} {\tau_h f} \normat {(\Lexp M)^*} g }{\normat {\Lexp M} f \le 1} \le \\ 2 \normat {(\Lexp M)^*} g \ . 
\end{multline*}
\item For each $\rho > 0$ we have found that 
  \begin{equation*}
    \expectat M {\Phi(\rho \tau_h f)} \le \frac1{\sqrt 2}\euler^{\frac{\absoluteval h^2}2} \left(\int \Phi(2\rho f(z)) M(z)\ dz\right)^{\frac12}  
  \end{equation*}
where the right-end-side if finite for all $\rho$ if $f \in \Cexp M$. It follows that $\tau_h f \in \Cexp M$. Recall that $f \in \Ccs{\reals^n}$, implies $\tau_h f \in \Ccs{\reals^n}$ and $\lim_{h\to0} \tau_h f = f$ in the uniform topology. Let $f_n$ be a sequence in $\Ccs{\reals^n}$ that converges to $f$ in $\Lexp M$-norm. Let $\absoluteval h \le \sqrt{\log 2}$ and let $A$ be positive and $\Phi(A)=1$.
\begin{multline*}
  \normat {\Lexp M} {\tau_h f - f} = \\ \normat {\Lexp M} {\tau_h(f - f_n) + (\tau_h f_n - f_n) - (f-f_n)} \le \\
\normat {\Lexp M} {\tau_h(f - f_n)} + \normat {\Lexp M} {\tau_h f_n - f_n} + \normat {\Lexp M} {f-f_n} \le \\ 2 \normat {\Lexp M} {f - f_n} + A^{-1} \normat {\infty} {\tau_h f_n - f_n} + \normat {\Lexp M} {f-f_n} \le \\ 3 \normat {\Lexp M} {f - f_n} + A^{-1} \normat {\infty} {\tau_h f_n - f_n} \ ,
\end{multline*} 
which implies the desired limit at 0. The continuity at a generic point follows from the continuity at 0 and the semigroup property,

\begin{equation*}
\lim_{k \to h}\normat {\Lexp M} {\tau_{k} f - \tau_{h} f} = \lim_{k - h  \to 0} \normat {\Lexp M} {\tau_{k-h}(\tau_{h} f) - \tau_h f} = 0 \ .
\end{equation*}
\end{enumerate}
\ \qed\end{proof}

We conclude by giving, without proof, the corresponding result for a translation by a probability measure $\mu$, namely $\tau_\mu f(x) = \int f(x-y) \mu(dy)$. We denote by $\probabilities_{\euler}$ the set of probability measures $\mu$ such that $h \mapsto \euler^{\frac12 \absoluteval h^2}$ is integrable for example, $\mu$ could be a normal with variance $\sigma^2I$ and $\sigma^2 < 1$, or $\mu$ could have a bounded support.

\begin{proposition}[Translation by a probability]
\label{prop:taubymu} Let $\mu \in \probabilities_{\euler}$.   
\begin{enumerate}
\item The mapping $f \mapsto \tau_\mu f$ is linear and bounded from $\Lexp M$ to itself. If, moreover, $\int \euler^{\absoluteval h^2/2}\ \mu(dh) \le \sqrt 2$, then its norm is bounded by 2.
\item \label{item:taubymu2} If $f \in \Cexp M$ then $\tau_\mu f \in \Cexp M$. The mapping $\probabilities_\euler \colon \mu \mapsto \tau_\mu f$ is continuous at $\delta_0$ from the weak convergence to the $\Lexp M$ norm.
  \end{enumerate}
\end{proposition}

We can use the previous proposition to show the existence of sequences of mollifiers. A bump function is a non-negative function $\omega$ in $\Cinfcs{\reals^n}$ such that $\int \omega(x) \ dx = 1$. It follows that $\int \lambda^{-n} \omega(\lambda^{-1} x) \ dx = 1$, $\lambda > 0$ and the family of mollifiers $\omega_\lambda(dx) = \lambda^{-n} \omega(\lambda^{-1} x)dx$ converges weakly to the Dirac mass at 0 as $\lambda \downarrow 0$, so that for all $f\in\Cexp M$, the translations $\tau_{\omega_\lambda} f \in \Cinfcs{\reals^n}$ and convergence to $f$ in $\Lexp M$ holds for $\lambda\to0$ .

\section{Conclusions}
\label{sec:conclusions}

We have discussed the density  for the bounded point-wise convergence of the space of smooth functions $\Cinfcs {\reals^n}$ in the exponential Orlicz space with Gaussian weight $\Lexp M$. The exponential Orlicz class $\Cexp M$ has been defined as the norm closure of the space of smooth functions. The continuity of translations holds in the latter space.

The continuity of translation is the first step in the study of differentiability in the exponential Gauss-Orlicz space. The aim is to apply non-parametric exponential models to the study of Hyv\"arinen divergence \cite{MR2249836,lods|pistone:2015} and the projection problem for evolution equations \cite{brigo|legland|hanzon:99,brigo|pistone:2017CIG}. A preliminary version of the Gauss-Orlicz-Sobolev theory has been published in the second part of \cite{lods|pistone:2015}.  

\subsubsection*{Acknowledgments} The author thanks Bertrand Lods (Universit\`a di Torino and Collegio Carlo Alberto, Moncalieri) for his comments and acknowledges the support of de Castro Statistics and Collegio Carlo Alberto, Moncalieri. He is a member of GNAMPA-INDAM.  
%
%
%
\bibliographystyle{splncs03}

\begin{thebibliography}{1}
\providecommand{\url}[1]{\texttt{#1}}
\providecommand{\urlprefix}{URL }

\bibitem{brigo|legland|hanzon:99}
Brigo, D., Hanzon, B., Le~Gland, F.: Approximate nonlinear filtering by
  projection on exponential manifolds of densities. Bernoulli  5(3),  495--534
  (1999)

\bibitem{brigo|pistone:2017CIG}
Brigo, D., Pistone, G.: Projection based dimensionality reduction for measure
  valued evolution equations in statistical manifolds. In: Nielsen, F.,
  Critchley, F., Dodson, C. (eds.) Computational Information Geometry. For
  Image and Signal Processing, pp. 217--265. Signals and Communication
  Technology, Springer (2017)

\bibitem{dellacherie|meyer:75}
Dellacherie, C., Meyer, P.A.: Probabilit\'es et potentiel. Chapitres I \`a IV.
  \'Edition enti\`erment refondue. Hermann (1975)

\bibitem{MR2249836}
Hyv\"arinen, A.: Estimation of non-normalized statistical models by score
  matching. J. Mach. Learn. Res.  6,  695--709 (2005)

\bibitem{lods|pistone:2015}
Lods, B., Pistone, G.: Information geometry formalism for the spatially
  homogeneous {B}oltzmann equation. Entropy  17(6),  4323--4363 (2015)

\bibitem{musielak:1983}
Musielak, J.: Orlicz spaces and modular spaces, Lecture Notes in Mathematics,
  vol. 1034. Springer-Verlag (1983)

\bibitem{pistone:2013GSI}
Pistone, G.: Nonparametric information geometry. In: Nielsen, F., Barbaresco,
  F. (eds.) Geometric science of information, Lecture Notes in Comput. Sci.,
  vol. 8085, pp. 5--36. Springer, Heidelberg (2013)

\bibitem{pistone|sempi:95}
Pistone, G., Sempi, C.: An infinite-dimensional geometric structure on the
  space of all the probability measures equivalent to a given one. Ann.
  Statist.  23(5),  1543--1561 (October 1995)

\end{thebibliography}

\end{document}